\documentclass[12pt]{amsart}

\usepackage{amsmath}
\usepackage{amsfonts}
\usepackage{amssymb}
\usepackage{verbatim}
\usepackage[left=2.9cm,right=2.9cm,top=3cm,bottom=3cm]{geometry}

\usepackage[utf8]{inputenc}
\usepackage[T1]{fontenc}

\usepackage{enumerate}

\usepackage{color}
\usepackage{graphicx}

\newcommand{\mE}{{\mathbb E}}

\newcommand{\calF}{{\mathcal F}}

\newcommand{\mP}{{\mathbb P}}

\newcommand{\mR}{{\mathbb R}}

\newcommand{\F}{\mathcal{F}}

\newcommand{\R}{\mathbb R}
\newcommand{\1}{\mathbf 1}

\newtheorem{theorem}{Theorem}
\newtheorem{lemma}{Lemma}
\newtheorem{prop}{Proposition}
\newtheorem{corollary}{Corollary}

\newtheorem{rem}{Remark}

\newtheorem{assumption}{Assumption}

\title[Estimating the interaction graph]{Estimating the interaction graph of stochastic neuronal dynamics by
  observing only pairs of neurons}

\author[E.\,De Santis]{E.\,De Santis}
\address{Dipartimento di Matematica, Universit\`a di Roma La Sapienza,
Piazzale Aldo Moro, 5, 00185, Rome, Italy}
\email{desantis@mat.uniroma1.it}

\author[A.\,Galves]{A.\,Galves}
\address{Instituto de Matem\'atica e Estat\'{i}stica, Universidade de S\~ao Paulo, Rua do Mat\~ao 1010, 05508-090, S\~ao
Paulo, Brazil.}
\email{galves@usp.br}

\author[G.\,Nappo]{G.\,Nappo}
\address{Dipartimento di Matematica, Universit\`a di Roma La Sapienza, Piazzale Aldo Moro, 5, 00185, Rome, Italy}
\email{nappo@mat.uniroma1.it}

\author[M.\,Piccioni]{M.\,Piccioni}
\address{Dipartimento di Matematica, Universit\`a di Roma La Sapienza, Piazzale Aldo Moro, 5, 00185, Rome, Italy}
\email{mauro.piccioni@uniroma1.it}

\date{June 20,  2021}
\begin{document}

\begin{abstract}
We address the questions of identifying pairs of interacting neurons from the observation of their spiking activity. The neuronal network is modeled by a system of interacting point processes with memory of variable length. The influence of a neuron on another can be either excitatory or  inhibitory. To identify the  existence and the nature of an interaction we propose an algorithm  based only on the observation of joint activity of the two neurons in successive time slots. This reduces the amount of computation and storage required  to run the algorithm, thereby making the algorithm suitable for the analysis of real neuronal data sets. We obtain computable upper bounds for the probabilities of false positive and false negative detection. As a corollary we  prove the consistency of the identification algorithm.


\end{abstract}

\maketitle

\section{Introduction} \label{intro}

We address the question of inferring the interactions in a system of spiking neurons modeled as follows. The spiking activity of each neuron is a point process whose intensity
depends on the previous activity of a  set of neighbors, henceforth called its presynaptic neurons. Each neuron $i$ is affected by the spiking activity of its presynaptic neurons taking place after the
last spiking time of $i$. This means that a neuron resets its memory after each of its spikes. This biologically motivated feature implies that this system of interacting point processes  has a memory of variable length, see Rissanen (1983) \cite{rissanen}.
In the present article we introduce a new statistical
procedure to infer for each pair of neurons, whether one is presynaptic to the other.

The model of interacting spiking neurons considered here was first introduced in Galves and L\"{o}cherbach (2013)~\cite{Galves-Loch:13} in a discrete time framework. Models in this class were subsequently analyzed in many articles, including \cite{AAEE}, \cite{duarte_hydrodynamic_2015}, \cite{fournier2016}, \cite{andre}, \cite{andre.planche}, \cite{monma},  \cite{Baccelli2021ThePL}, and \cite{yu2021metastable}.

As far as we know, the problem of inferring the graph of interactions  for such a kind of models has been addressed only  in Duarte et al.~(2019). Given a sample of the spiking activity of a  large set of neurons,
they propose a pruning procedure
to retrieve the set of presynaptic neurons of a fixed neuron $i$.
First the algorithm  assumes that all the remaining neurons are presynaptic to $i$.
Then the nonparametric
maximum likelihood estimates of the spiking probabilities of $i$ are computed, as a function of the spiking activities of the other neurons after the last spiking time of $i$. The same procedure is
repeated by excluding the candidate presynaptic neuron $j$.
The criterion to prune or not
$j$ from the estimated set of presynaptic neurons of $i$ is the
following. For any fixed observed history of the activities of the other neurons, the difference between the estimated
probabilities, with or without the information
concerning $j$, is computed. If the maximum of these differences is below a given threshold, then $j$ is pruned.
Otherwise, $j$ is kept in the estimated set of presynaptic neurons of $i$.
Under suitable assumptions, in \cite{dglo}  the consistency of the procedure has been obtained:  upper bounds are provided for the probabilities of false positive and
false negative detection, that converge to $0$ as the length of the sample increases.

Despite the clear mathematical interest of the result obtained in Duarte et al.~\cite{dglo}, the proposed procedure
presents drawbacks when used to analyze real neuronal data. First of all, it requires extremely lengthy data sets in order to observe the possible histories
of a big set of neurons for a sufficiently large number of times. Moreover, the required computations cannot be localized in the observed set of neurons. In fact, to obtain the required estimates, the histories of all the neurons
have to be taken into account at the same time. Even more questionable is the assumption that all the neurons of the system can be observed. Actually, the activity of many inhibitory neurons can hardly be observed directly, by means of the actual spike sorting procedures.

The algorithm introduced in the present article aims to overcome these drawbacks. To guess the influence of neuron $j$ on neuron $i$, only  the spiking activities of these two neurons are considered.
The observation time of the system is divided into short time slots. Then the following two probabilities are compared:  (\emph{a}) the probability of a spike of neuron $i$, following another spike of $i$, observed in the previous time slot; (\emph{b}) the probability of a spike of $i$, following a spike of $j$ and one of $i$, observed in the two previous time slots (first $i$ and then $j$).
 For sufficiently small time slots, the difference between the latter and the former probabilities reveals the eventual presence of $j$ in the set of presynaptic neurons of $i$. 
  Under suitable assumptions, in the limit as the length of the time slots decreases to $0$, if $j$ has an excitatory, respectively inhibitory, effect on~$i$, then this difference becomes positive, respectively negative.  If $j$ is not in the set of presynaptic neurons of $i$, then the  limit of this difference is 0. These asymptotic results provide the basis for the statistical algorithm considered in the present article.

To recover the limiting behavior described above from the observation of a large but finite sample, the length of the time slot must be sufficiently small. On the other hand,   we need to  observe  a sufficiently large number of times the event that two spikes of $i$, with a spike of $j$ in between, occur in three consecutive time slots. Therefore the length of the time slot cannot be  too small.
 This fact is reminiscent of the familiar bias-variance tradeoff.  As a matter of fact, in our proof we estimate the maximal length of the time slots for which our consistency proof works.

The rationale behind the algorithm  considered here is simple to explain. The algorithm considers the spiking activity of neurons $i$ and $j$, in three successive time slots, starting with a spike of $i$, in order to take advantage of the reset feature of the neuronal activity. Indeed, after each spike, a neuron resets its memory by forgetting the previous history of the system. Therefore, if we know that there was a spike of $i$ in the first time slot, the reset property helps detecting if a spike of neuron $j$, occurring in the second time slot, influences or not the activity of $i$  in the third time slot.

The algorithm discussed here is reminiscent of the approach to infer neuronal interactions introduced in the seminal papers by Brillinger and co-authors \cite{Brillinger1976}, \cite{Brillinger1988}. More recently, the problem of identifying pairs of interacting components in a different class of multivariate point processes, namely Hawkes processes, was addressed by Eichler, Dahlhaus, Dueck (2017) \cite{Eichler15graphicalmodeling}, in the framework of Granger causality (see Granger 1969~\cite{granger}).

The structure of the paper is the following. In Section 2 the model is introduced and the two main results are stated. The proofs of Theorems \ref{differenz} and \ref{main} are given in Sections~3 and 4, respectively.

\section{Definitions and main result}

We start by introducing the multivariate point process modeling the system of spiking neurons considered here. The main ingredients used to define the
process are the following:

\begin{itemize}
\item  a finite set $I$, henceforth called the set of \textit{neurons};
\item  a matrix $(w_{j \to i} \in  \R : (j,i) \in I^2 )$,  henceforth called the  matrix of \textit{synaptic weights};
\item a family of  simple point processes $\{ \left (T^{i}_{n }\right)_{n \ge 1}: i\in I\}$, with $0 <T^{i}_1 < T^{i}_2<\ldots $ , denoting the successive spiking times of neuron $i$;
\item a family of  non-decreasing functions  $\phi_i : \R  \to [ 0, +\infty [$, henceforth called {\it spiking rate functions}.
\end{itemize}

If $w_{j \to i } > 0 $ (respectively $w_{j \to j } < 0 $), we say that the neuron $j$ has an {\it excitatory} (respectively {\it inhibitory}) effect on neuron $i$.
In case $w_{j \to i }= 0 $, we say that neuron $j$ does not affect neuron $i$. We assume that there is no self-interaction and therefore  $w_{j \to j } = 0 $, for all $j$.  The reason for this terminology will be readily clarified  (see \eqref{def:U-i} and \eqref{def:intensity-N-i}).

The set $\mathcal {V}^i=\{j: w_{j \rightarrow i} \neq 0\}$, is called the set of \textit{presynaptic neurons} of $i$. Obviously
\[
\mathcal {V}^i=\mathcal V^{i}_{+} \cup \mathcal V^{i}_{-}\, ,
\] 
where $V^{i}_{+}$ and $V^{i}_{-}$ are the sets
of excitatory and  inhibitory ones, 
$$
\mathcal V^{i}_{+} =\left \{j: w_{j \rightarrow i} >0\right \},\qquad \mathcal V^{i}_{-} =\left \{j :w_{j \rightarrow i} <0 \right \},
$$
respectively. 

We use the notation $d$ to denote the maximum cardinality of the sets of presynaptic neurons
\begin{equation*}
d=\max\{|\mathcal V^{i}|; i \in I \}\, .
\end{equation*}

\vskip 6truemm

For any neuron $i$, we define the spike counting measure $N^i$ as follows. For any subset $A\subset \mathbb{R}^+$,
\[
N^{i}(A)=\sum_{n\ge 1}\1_{\{ T^{i}_n \in A\}}.
\]
For any positive real number $t$, when  the event $\{T^{i}_{1} < t\}$ is realized, we define $L^{i}(t)$  as the last spiking time of neuron $i$ occurring before time $t$
\[
L^{i}(t)=\sup\{n \ge 1 : T^{i}_{n} < t \}\, .
\]
This  definition allows to introduce the \textit{membrane potential} $U^{i}(t)$ of neuron $i$ at time $t$ as follows
\begin{equation}\label{def:U-i}
U^i(t)=
\begin{cases}
 U^i(0)+\sum_{j  \in \mathcal {V}^i} w_{j \to i} N^{j}(0, t],, &\text{ if } 0\leq t< T^{i}_{1},
\vspace{2mm}
 \\
\sum_{j  \in \mathcal {V}^i} w_{j \to i} N^{j}(L^{i}(t), t], &\text{ if} \, t\geq T^{i}_{1},
\end{cases}
\end{equation}
where $U^{i}(0)$ denotes the initial value of the membrane potential.

We will denote by $U(t)$ the vector of the membrane potentials of all the neurons at time $t$
\[
U(t)= \left(U^{i}(t): i \in I \right).
\]
In what follows, the initial value $U(0)$ of the vector  of membrane potentials is chosen in an arbitrary way. We are not assuming the stationarity of the processes.

Finally, for any positive real number $t$, we define $\F_t$ as the $\sigma-$algebra  generated by the family of spike counting measures $\left( N^{i}(A): i \in I, A \subset [0,t] \right) $,
together with the initial vector of membrane potentials $U(0)=\left(U^i(0): i \in I \right)$.

 \vskip 6truemm

With this notation, we can now formally relate the elements of the  model in the following way.
For any neuron $i$ and any pair of positive real number $t <t^{\prime}$, we require the spike counting measures $N^{i}(t, t^{\prime}]$ , $i \in I$, to satisfy the equation
\begin{equation}\label{def:intensity-N-i}
\mE\left( N^{i}(t,t^{\prime}]   \, |\,   \F_t \right)= \mE\left(\int_t^{{t^{\prime}}}\phi_i (U^i(r)) dr |  \F_t  \right)\, .
\end{equation}
Informally this condition can be stated as
\[
\mP\left( N^{i}(t,t+dt] =1 \, |\,   \F_t \right)= \phi_i (U^{i}(t))dt + o(dt).
\]

 \vskip 6truemm

\begin{assumption}\label{phi} The spiking rate functions $\phi_i: \mathbb {R} \to (0,+\infty)$ are 
\begin{enumerate}
 \item nondecreasing,
\item
bounded away from $0$, with
$$
\alpha=\min_{i \in I}\inf_{u \in \mR}\phi_i(u)>0\, ,
$$
\item  bounded above, with
$$
\beta= \max_{i \in I}\sup_{u \in \mR}\phi_i(u)< +\infty\, ,
$$
\item and satisfy
$$
\min_{i \in I} \{| \phi_i(w_{j \to i})-\phi_i(0)|: j \in \mathcal {V}^i\}=\delta >0.
$$
\end{enumerate}
\end{assumption}
Obviously, $\alpha$, $\beta$ and $\delta$ are such that $\alpha+\delta\le \beta$.
In the sequel we shall use the shorthand notation
\begin{equation}\label{s-tau}
s=\frac {\alpha}{\beta}\, \, \, \mbox{and\, } \tau=\frac {\delta}{\beta} \in (0,1),
\end{equation}
which are constrained by  $s + \tau \le 1$.

\vskip 6truemm

Before defining the estimation algorithm, we need to introduce the following events, depending on a parameter $\Delta >0$ to be chosen in a suitable way.
Given two neurons $i \in I$ and $j \in I$, with $i \neq j$,  we  denote
\begin{align}\label{eventoA0}
A^{ i }(\Delta)&= \{ N^{i}(0 , \Delta  ] >0 \},
\\\label{eventoB0}
B^{ i }(\Delta)&= A^{ i }(\Delta) \cap \{  N^{i}(\Delta, 2\Delta  ] >0   \} ,
\\\label{eventoC0}
C^{ j\to i }(\Delta)&= A^{ i }(\Delta) \cap \{ N^j  (\Delta, 2 \Delta ] >0 \} ,
\\\label{eventoD0}
D^{ j\to i }(\Delta)&=C^{ j\to i }(\Delta) \cap \{N^i(2\Delta, 3 \Delta] >0\}.
\end{align}

In the following, we are going to use the notation  $\mP_{u}(\cdot)$ instead of $\mP(\cdot \, | \, U(0)=u)$.

\begin{theorem}\label{differenz} Suppose that the family of spiking rate functions  $\{\phi_i: i \in I \}$ satisfy Assumption \ref{phi}.   Let
$\Delta^*=\frac {s^3\tau}{34 d\beta}$.
 Then,  for any value of
$\Delta \in (0, \Delta^*]$, any fixed pair of neurons $i$ and $j$ with $i \neq j$, and any pair of vectors of membrane potentials $u$ and $u^{\prime}$,
the following inequalities hold.

\indent If $j \notin \mathcal {V}^i$, then
\begin{equation}\label{both}
- \xi_1(\Delta) < \frac { \mP_u (D^{ j\to i }(\Delta))}{ \mP_u (C^{ j\to i }(\Delta))}-\frac {\mP_{u^{\prime}} (B^{ i }(\Delta))}{ \mP_{u^{\prime}} (A^{ i }(\Delta))} < \xi_2(\Delta).
\end{equation}
\indent If $j \in \mathcal {V}^i_-$, then
\begin{equation}\label{inhibitory}
\frac { \mP_u (D^{ j\to i }(\Delta))}{ \mP_u (C^{ j\to i }(\Delta))}-\frac {\mP_{u^{\prime}} (B^{ i }(\Delta))}{ \mP_{u^{\prime}} (A^{ i }(\Delta))}
\leq - \xi_1(\Delta).
\end{equation}
\indent If $j \in \mathcal {V}^i_+$, then
\begin{equation}\label{excitatory}
\xi_2(\Delta) \leq
\frac { \mP_u (D^{ j\to i }(\Delta))}{ \mP_u (C^{ j\to i }(\Delta))}-\frac {\mP_{u^{\prime}} (B^{ i }(\Delta))}{ \mP_{u^{\prime}} (A^{ i }(\Delta))} \, ,
\end{equation}
where
\begin{align}\label{csi1}
\xi_1(\Delta)&=\beta \Delta \left[\frac {1}{5}\tau+(9-\frac {\tau}{10})\frac {d\beta\Delta}{s^2}\, \right]
\intertext{and}
\label{csi2}
\xi_2(\Delta)&=\beta \Delta \left\{\frac {1}{5}\tau+\left[5+3s^2+\frac {\tau}{10}(5-3s^2)\right]\frac {d\beta\Delta}{s^3}\right\}.\
\end{align}

\end{theorem}

\noindent In the rest of this section, we will always take $\Delta=\Delta^{*}$.\\

\noindent Theorem \ref{differenz} suggests estimation algorithm based on a partition of the observation time in slots of fixed length $\Delta^{*}$.  Given two neurons $i \in I$ and $j \in I$, with $i \neq j$, we set
$$
A^{ i }_1=A^{ i }(\Delta^{*}), \quad B^{ i }_1=B^{ i }(\Delta^{*}), \quad C^{j\to i }_1=C^{j\to i }(\Delta^{*}), \quad  D^{j\to i }_1=D^{j\to i }(\Delta^{*}),
$$
and likewise, for any positive integer $k> 1$, we  define the events
\begin{equation*}%\label{eventoA0}
A^{ i }_k= \{ N^{i}((2k-2)\Delta^{*} ,(2k-1)\Delta^{*}  ] >0 \},
\end{equation*}
\begin{equation*}%\label{eventoB0}
B^{ i }_k= A^{ i }_k \cap \{  N^{i}((2k-1)\Delta^{*}, 2k\Delta^{*}  ] >0   \} ,
\end{equation*}
\begin{equation*}%\label{eventoC0}
C^{ j\to i }_{k}= \{ N^i((3k-3)\Delta^{*}, (3k-2)\Delta^{*}]>0, \; N^j  ((3k-2)\Delta^{*}, (3k-1)\Delta^{*}  ] >0 \} ,
\end{equation*}
\begin{equation*}%\label{eventoD0}
D^{ j\to i }_{k}=C^{ j\to i }_{k} \cap \{N_i((3k-1)\Delta^{*}, 3k \Delta^{*}] >0\}.
\end{equation*}

For any integer $n \ge 1$, we define 
$$
S^{ A^i}(n)=\sum_{k=1}^n \mathbf{ 1}_{A_{k}^i }, \quad
 S^{B^i}(n)=\sum_{k=1}^n \mathbf{ 1}_{B_{k}^{i}},
 $$
 $$
 S^{C^{j\to i}}(n)=\sum_{k=1}^n \mathbf{ 1}_{C_{k}^{j\to i} },\quad
 S^{ D^{j\to i}}(n)=\sum_{k=1}^n \mathbf{ 1}_{D_{k}^{j\to i} }\,,
$$
and, for any  integer $m\geq 1$, we define 
\[
K^i_m= \inf  \{ n \ge 1 : S^{ A^i}(n)= m\},
\]
\[
H^{j\to i}_m= \inf  \{ n \ge 1 : S^{C^{j\to i}}(n) =m \}.
\]

Now suppose that the processes are observed up to a time horizon $T=3\Delta^{*}n$ and define
\begin{equation}\label{parameters}
t_n=\left\lceil \alpha \Delta^{*}n \right\rceil ,\, \, m_n=\left\lceil \frac {19}{20}  \, \alpha^2(\Delta^{*})^2(1-\frac {\tau}{10} \sqrt {\alpha \Delta^{*}})n\right\rceil\, .
\end{equation}
\bigskip

Once these parameters are set, we can define the empirical ratios
\begin{align}\label{erre-1-2}
R^i(n)&=\begin{cases}
\dfrac{S^{B^i}(K^i_{m_n})}{m_n}\ , & \mbox{\, if\, } K^i_{m_n}\le t_n
\vspace{2.5mm}
\\
\dfrac{S^{B^i}(t_n)}{S^{A^i}(t_n)}, & \mbox{\, if\, } K^i_{m_n} > t_n,
\end{cases}
\intertext{and}
\label{gi-1-2}
G^{j\to i}(n)&=
\begin{cases} \dfrac{S^{D^{j\to i}}(H^{j\to i}_{m_n})}{m_n}\ , & \mbox{\, if\, } {H^{j\to i}_{m_n}} \le n
\vspace{2.5mm}
\\
\dfrac{S^{D^{j\to i}}(n)}{S^{C^{j\to i}}(n)}\ ,& \mbox{\, if\, } {H^{j\to i}_{m_n}} > n.
\end{cases}
\end{align}

For any pair of neurons $i\neq j$,  the statistics $R^i(n)$
and $G^{j\to i}(n)$ will be used to identify whether $j$ is presynaptic to $i$ or not. They are ratio estimators that are stopped once their denominator reaches the level $m_n$, but differently from $G^{j \to i}(n)$, which is allowed to reach this level up to the expiration of the whole time horizon $T=3\Delta^{*}n$, the estimator $R^i(n)$ is stopped at most after the time $2\Delta^*t_n$ is expired. Since a rough upper bound for $\alpha \Delta^{*}$ is $\frac{1}{34d}$, this time is much smaller than $T$: therefore a much smaller interval of time is sufficient for $R^i(n)$ to reach the same accuracy as $G^{j \to i}(n)$.

Inspired by Theorem \ref{differenz}  we define the estimated sets
$\hat{\mathcal {V}}^{i}_{+}(n)$, $\hat {\mathcal {V}}^{i}_{- }(n)$, and $\hat {\mathcal {V}}^{i}(n)$, as follows
\begin{equation}\label{V-i-minus}
\hat {\mathcal {V}}^{i}_{- }(n)=\{j \in I \setminus \{i\}:  G^{j\to i}(n) -R^i(n) \leq -  \xi_1(\Delta^*) \}\,
\end{equation}
\begin{equation}\label{V-i-plus}
\hat{\mathcal {V}}^{i}_{+}(n)=\{j \in I\setminus \{i\}:  G^{j\to i}(n) - R^{i}(n) \geq  \xi_2 (\Delta^*) \}\, , \,
\end{equation}
\begin{equation}\label{V-i}
\hat {\mathcal {V}}^i(n)= \hat{\mathcal {V}}^{i}_{+}(n) \cup \hat {\mathcal {V}}^{i}_{- }(n)\, .
\end{equation}

We can now state our main theorem, in which $\mP$ stands for a probability measure on the spiking processes with an arbitrary law of the vector of the membrane potentials.

\begin{theorem}\label{main} Let $T=3 \Delta^* n$, where  $\Delta^{*}= \frac {s^3\tau}{34 d\beta}$ and $n$ is a positive integer. Let
$R^i(n)$ and $G^{j\to i}(n)$ be the estimators defined in equations \eqref{erre-1-2} and \eqref{gi-1-2}, and $\hat {\mathcal {V}^{i}}(n)$, $\hat {\mathcal {V}^{i}}_{-}(n)$ and $\hat {\mathcal {V}^{i}}_{+}(n)$ be defined as in \eqref{V-i}, \eqref{V-i-minus}, \eqref{V-i-plus}. Then the following inequalities hold:
$$
\textit{if     } j \notin \mathcal {V}^{i} \textit{     then    } \mP \big(j \not \in  \hat {\mathcal {V}}^{i}(n)\big) \geq 1-6\-e^{-\omega T}  ;
$$
$$
\textit{if     } j \in \mathcal {V}^i_{-}\textit{     then    } \mP \big(j \in  \hat {\mathcal {V}}^{i}_{-}(n)\big) \geq 1-4\-e^{-\omega T} ;
$$
$$
\textit{if     } j \in \mathcal {V}^{i}_{+ }\textit{     then    } \mP \big(\hat {j \in \mathcal {V}^i}_{+}(n)\big) \geq 1- 4\-e^{-\omega T},
$$
where
$$
\omega = \vartheta_0 \, \frac{\tau^4 s^{9} \beta}{d^2},
$$
$\vartheta_0$ being a computable universal constant.
\end{theorem}

\begin{rem}  
From \eqref{s-tau} the constant $ \frac{\tau^4 s^{9} \beta}{d^2}$ can be rewritten as $\frac{\delta^4\alpha^9}{d^2 \beta^{12}}$.
When only  positive lower bounds are available for the parameters $\alpha$ and $\delta$, and only  upper bounds  for  $\beta$ and~$d$, at least a lower bound for $\omega$  can be obtained.
\end{rem}

\section{Proof of Theorem \ref{differenz}}\label{proof1}

The proof is based on a particular construction of the counting measures $N^{i}$, and the membrane potential processes $U^i$, $ i\in I$, in such a way that \eqref{def:U-i}  and \eqref{def:intensity-N-i}  hold.

 We consider  a  Poisson measure on $[0,\infty)\times [0,\beta]\times I$,
 $$
 \mathcal{N}(dt, dx,dz)= \sum_{k\geq 1} \delta_{(\mathcal{T}_k, X_k,Z_k)}(dt,dx,dz), 
 $$
 with intensity measure
 $$
 \mu_{\mathcal{N}}(dt\times dx\times \{i\})=dt\times dx,
 \quad \text{$i\in I$.}
 $$
 Without loss of generality we assume that  $0<\mathcal{T}_1<\mathcal{T}_2<...$. 
  Observe that the 
  marks $X_k$, $k\geq 1$,  are independent and uniform in $[0, \beta]$.
  The  $\mathcal{T}_k$'s are  candidates to be spiking times for neurons:
 $\mathcal{T}_1$ is accepted as a spike for neuron $j$ if and only if
$$
Z_1=j, \quad X_1\leq \phi_j(U^j(0)).
$$
If this is the case then $T^j_1=\mathcal{T}_1$, and  the potential vector is updated in the interval $[\mathcal{T}_1,\mathcal{T}_2)$ according to
\begin{align*}
U^j(t)=0, \quad
U^i(t)=  U^i(0)+w_{j\to i}, \quad i\neq j;
\end{align*}
otherwise
$U(t)=U(0)$, 
and the construction proceeds with the next candidate time.

Recursively in $k$ the candidate time
 $\mathcal{T}_k$ is accepted as a spike for neuron $Z_k$ if and only if
$$
 \quad X_k\leq \phi_{Z_k}(U^{Z_k}(\mathcal{T}_k)),
$$
If  this is the case the potential vector  is updated in the interval $[\mathcal{T}_k,\mathcal{T}_{k+1})$ according to
\begin{align*}
U^{Z_k}(t)=0, \quad U^i(t)=  U^i(\mathcal{T}_k)+w_{Z_{k}\to i}, \quad i\neq Z_k;
\end{align*}
otherwise
 $U(t)=U(\mathcal{T}_k)$, and the construction proceeds with the next candidate time.

Next, for each neuron $i \in I$, and $x\in (0,\beta]$, define the  Poisson measures   $N^{i,x}(t,t^\prime]:= \mathcal{N}((t,t^\prime]\times [0,x]\times\{i\}),$ with intensity measure $x\, dt$, on $(0,\infty)$.

From now on we set
$$
\overline{N}^i=N^{i,\beta}, \quad \underline{N}^i=N^{i,\alpha}
$$
so that for any $i\in I$
$$
\underline{N}^{i}(t,t^\prime]\leq N^{i}(t,t^\prime]\leq \overline{N}^{i}(t,t^\prime], \quad 0\leq t< t'.
$$
It is also convenient to use the notation
$$
N^{\mathcal {W}}(t,t^\prime]=\sum_{j \in \mathcal {W}}N^j(t,t^\prime],\quad  \quad 0\leq t< t', \quad \mathcal {W}\subset I,
$$
and the same notation for the measures $\underline{N}^j$, and  $\overline{N}^j$.

For each pair of neurons $i\neq j$, and for each positive real number $\Delta$ we now define the events
$$
\tilde A^i(\Delta)=A^i(\Delta)\cap\left\{N^{\mathcal {V}^i}(0,2\Delta]=0\right\},
$$
$$
\tilde B^i(\Delta)=B^i(\Delta) \cap\left\{ N^{\mathcal {V}^i}(0,2\Delta]=0\right\},
$$

$$
\tilde {C}^{j\to i}(\Delta)= {C}^{j\to i}(\Delta)\cap\left\{N^{\mathcal {V}^i}(0,\Delta]=N^{\mathcal {V}^i \setminus \{j\}}(\Delta, 2\Delta]=N^{\mathcal {V}^i}(2\Delta, 3\Delta]=0\right\},
$$
$$
\tilde{D}^{j\to i}(\Delta)={D}^{j\to i}(\Delta)\cap \left\{ N^{\mathcal {V}^i}(0,\Delta]=N^{\mathcal {V}^i \setminus \{j\}}(\Delta, 2\Delta]=N^{\mathcal {V}^i}(2\Delta, 3\Delta]=0\right\},
$$
where $A^i(\Delta)$, $B^i(\Delta)$, ${C}^{j\to i}(\Delta)$, and ${D}^{j\to i}(\Delta)$ have been defined in \eqref{eventoA0}, \eqref{eventoB0}, \eqref{eventoC0}, and \eqref{eventoD0}, respectively.

\begin{lemma}\label{sopraesotto}
 Irrespectively of the vector $u$ of membrane potentials, for any $\Delta>0$, 
the following inequalities hold.
\begin{equation}\label{pBAtilde}
\frac{\mP_u(\tilde B^i(\Delta))}{\mP_{u} (\tilde A^i(\Delta) )}
\left( 1- \frac{1-\-e^{-2d \beta  \Delta }}{s}\right)
\leq
\frac{\mP_u(B^i(\Delta))}{\mP_u (A^i(\Delta))}
  \leq
\left(1+\frac{\-e^{2d \beta  \Delta }-1}{s^{2}}\right)
\frac{\mP_u( \tilde B^i(\Delta))}{\mP_u(\tilde A^i(\Delta) )}
\end{equation}

\begin{equation}\label{pCDtilde}
 \frac{\mP_u(\tilde{D}^{j\to i}(\Delta))}{\mP_u (\tilde{C}^{j\to i}(\Delta) ) }
\left (1- \frac{1-\-e^{-3d \beta  \Delta }}{s^{2}}\right)
\leq
\frac{\mP_u(D^{j\to i}(\Delta))}{\mP_u ( C^{j\to i}(\Delta))}
\leq
\left (1+\frac{\-e^{3d \beta  \Delta }-1}{s^{3}}\right)
\frac{\mP_u(\tilde{D}^{j\to i}(\Delta))}{\mP_u (\tilde{C}^{j\to i}(\Delta) ) }
\end{equation}
\end{lemma}

\begin{proof}
First observe that for any triple of events $E$, $F$ and $G$ such that
$E\subset F$  the following inequalities hold
$$
\frac{\mP (E\cap G)}{\mP(F\cap G)}\, \left(1- \frac{\mP(F\setminus G)}{\mP(F)} \right) \leq \frac{\mP (E)}{\mP(F)}
\leq
\left( 1+\frac{\mP (E\setminus G)}{\mP(E\cap G)} \right)\, \frac{\mP (E\cap G)}{\mP(F\cap G)}.
$$
Choosing
$$E=B^i(\Delta),\quad F= A^i(\Delta), \quad G=\{N^{\mathcal {V}^i}(0,2\Delta]=0\}
$$
first, and  then
$$
E=D^{j\to i}(\Delta),\; F= C^{j\to i}(\Delta),\; G=\{N^{\mathcal {V}^i}(0,\Delta]=N^{\mathcal {V}^i \setminus \{j\}}(\Delta, 2\Delta]=N^{\mathcal {V}^i}(2\Delta, 3\Delta]=0\},
$$
 the inequalities \eqref{pBAtilde} and \eqref{pCDtilde} are consequences of the bounds
$$
\frac{\mP_u   ( A^i(\Delta) \setminus \tilde{A}^i(\Delta)  )   }{\mP_u( A^i(\Delta))}
\leq
\frac{  (1-\-e^{-2d \beta  \Delta }) (1- \-e^{-\beta  \Delta})
 }{   1- \-e^{-\alpha \Delta}}
\leq
\frac {\beta}{\alpha}(1-\-e^{-2d \beta  \Delta }),
$$
and
$$
\frac{\mP_u   (B^i(\Delta) \setminus \tilde{B}^i(\Delta) )   }{\mP_u(\tilde{B}^i(\Delta)) }
\leq \-e^{2d \beta  \Delta }\left(\frac{    1- e^{- \beta \Delta }     }{    1- e^{- \alpha \Delta }   }\right)^2(1-\-e^{-2d \beta  \Delta })
\leq \left( \frac {\beta}{\alpha} \right)^2(\-e^{2d \beta  \Delta }-1).
$$
 The former is due to
$$ 
A^i(\Delta) \setminus \tilde{A}^i(\Delta)\subset \{\overline N^i(0,\Delta]>0, \overline N^{\mathcal {V}^i}(0,2\Delta]>0\},\,\,\, \{\underline N^i(0, \Delta]>0\} \subset A^i(\Delta), 
$$ and the majorization
$$
\frac{   1-e^{-  \beta x }   } {   1-e^{-  \alpha x }   }\leq \frac {\beta}{\alpha},\,\, x\geq 0,
$$
valid as long as $\beta \geq \alpha$, in view of the fact that  the function $x\mapsto \frac{e^x-1}{x}$   is increasing in $\mathbb{R}$. The latter is due to
$$ 
B^i(\Delta) \setminus \tilde{B}^i(\Delta)\subset \{\overline N^i(0,\Delta]>0, \overline N^i(\Delta, 2\Delta]>0, 
\overline N^{\mathcal {V}^i}(0,2\Delta]>0\}
$$
and 
$$
\{\overline N^{\mathcal {V}^i}(0,2\Delta]=0,\,\,\, \underline N^i(0, \Delta]>0, \underline N^i(\Delta, 2\Delta]>0\} \subset {\tilde B}^i(\Delta). 
$$

With similar arguments one proves that

$$
\frac{\mP_u   (C^{j\to i}(\Delta)\setminus \tilde{C}^{j\to i}(\Delta) )   }{\mP_u( C^{j\to i}(\Delta)) }
\leq
\frac{   (1-e^{-  \beta \Delta })^2(1-\-e^{-3d \beta  \Delta })   }{ (1-e^{-  \alpha \Delta })^2 } \leq \left (\frac {\beta}{\alpha}\right)^2(1-\-e^{-3d \beta  \Delta })
$$
$$
\frac{\mP_u   ( D^{j\to i}(\Delta)\setminus \tilde{D}^{j\to i}(\Delta) )   }{\mP_u (\tilde{D}^{j\to i}(\Delta)) }
 \leq \-e^{3d \beta  \Delta }\left(\frac{   1-e^{-  \beta \Delta }    }{ 1-e^{-  \alpha \Delta } }\right)^3
(1-\-e^{-3d \beta  \Delta }) \leq \left ( \frac {\beta}{\alpha}\right )^3(\-e^{3d \beta  \Delta }-1).
$$
\end{proof}

\begin{lemma}\label{stimasugiu}
	Irrespectively of the  vector $u$ of membrane potentials, for any $\Delta >0$, one has
	\begin{equation}\label{PBA}
	\big( 1-\-e ^{- \phi_i (0) \Delta}\big)\-e^{-d(\beta-\alpha)\Delta}\leq
\frac{ \mP_u \big( \tilde{B}^{i}(\Delta)\big)}{\mP_u \big(\tilde{A}^{i}(\Delta)  \big)} \leq
\big(1-\-e ^{- \phi_i (0) \Delta}\big) \-e^{d(\beta-\alpha)\Delta}.
	\end{equation}
Moreover, provided $j \notin \mathcal {V}^i$
\begin{equation}\label{PDC}
	\left( 1-\-e ^{- \phi_i (0) \Delta}\right)\-e^{-d(\beta-\alpha)\Delta} \leq
\frac {\mP_u \big( \tilde{D}^{j\to i }(\Delta)\big)}{\mP_u \big( \tilde{C}^{j\to i }(\Delta)   \big)} \leq
\left (1-\-e ^{- \phi_i (0) \Delta}\right) \-e^{d(\beta-\alpha)\Delta},
	\end{equation}
whereas if $j \in \mathcal {V}^i_+$
	\begin{equation}\label{PDC>}
  \frac {\mP_u \big( \tilde{D}^{j\to i }(\Delta)\big)}{\mP_u \big( \tilde{C}^{j\to i }(\Delta)   \big)}  \geq
(1-\-e^{- (\phi_i (0)+\delta )\Delta})\-e^{-d(\beta-\alpha)\Delta},
	\end{equation}
and if $j \in \mathcal {V}^i_-$
	\begin{equation}\label{PDC<}
  \frac {\mP_u \big( \tilde{D}^{j\to i }(\Delta)\big)}{\mP_u \big( \tilde{C}^{j\to i }(\Delta)   \big)} \leq
(1-\-e^{- (\phi_i (0)   - \delta )\Delta}) \-e^{d(\beta-\alpha)\Delta}.
	\end{equation}
\end{lemma}

\begin{proof} 
The statements of the theorem will be proved with the sharper bounds obtained by replacing $d$ with $d_i=|\mathcal{V}^i|$.

It is convenient to split the events $\tilde{A}^i(\Delta)$ and $\tilde{C}^{j\to i}(\Delta)$ in the following way: $\tilde{A}^i(\Delta)=A^i_- \cap A^i_+$ and $\tilde{C}^{j\to i}(\Delta)=C^{ji}_-\cap C^i_+$, where
$$
A^i_-=\{N^i(0, \Delta]>0, N^{\mathcal {V}^i}(0,\Delta]=0\}, A^i_+=\{ N^{\mathcal {V}^i}(\Delta, 2\Delta]=0\},
$$
$$
C^{ji}_-=\{N^i(0, \Delta]>0,
N^j(\Delta,2\Delta]>0, N^{\mathcal {V}^i}(0,\Delta]=N^{\mathcal {V}^i \setminus \{j\}}(\Delta,2\Delta]=0\},
$$
$$
 C^i_+=\{N^{\mathcal {V}^i}(2\Delta, 3\Delta]=0\}.
$$

 Since for any triple of events $E$, $F_+$, $F_-$, such that $E \subset F=F_+\cap F_-$
$$
\frac{ \mP\left(E\right)}{\mP \left(F\right)}
=\mP \left( E| F \right)
= \mP\left( E |  F_{-} \cap F_{+} \right)
=\frac {\mP \left (E \cap F_+| F_-\right)}{\mP \left(F_+|F_- \right)},
$$
we get
$$
\frac{ \mP_u \big( \tilde B^{i}(\Delta)\big)}{\mP_u \big(\tilde A^i(\Delta)   \big)}=\frac {\mP_u \big(\tilde B^i(\Delta)\cap A^i_+| A^i_-\big)} {\mP_u \big(A^i_+|A^i_- \big)},
$$
$$
\frac {\mP_u \big( \tilde D^{j\to i}(\Delta)\big)}{\mP_u \big( \tilde C^{j\to i}(\Delta)   \big)}
=\frac {\mP_u \big(\tilde D^{j\to i}(\Delta) \cap C^i_+| C^{ji}_-\big)}{\mP_u \big(C^i_+|C^{ji}_-\big)}.
$$
Next we proceed to bound the right hand side in the previous formulas.

By writing $\mP^{k\Delta}_{v}(\cdot)=\mP(\cdot|U(k\Delta )  = v)$ for $k=1,2$, then by the Markov property we have
\begin{equation}\label{MarkovAB}
\frac {\underset{v:v^i =0 }{\inf}
 \mP^{\Delta}_{v} \left  ( N^i (\Delta,  2\Delta ] >0, A^i_+ \right)}{\underset{v:v^i =0 }{\sup}
 \mP^{\Delta}_{v} \big( A^i_+ \big)}
\leq \frac {\mP_u \big(\tilde B^{i}(\Delta) \cap A^i_+| A^i_-\big)} {\mP_u \big(A^i_+|A^i_- \big)} \leq \frac{\underset{v:v^i =0 }{\sup}\mP^{\Delta}_{v} \left(  N^i (\Delta,  2\Delta ] >0, A_{ i }^{+} \right)}{\underset{v:v^i =0 }{\inf}\mP^{\Delta}_{v} \left  ( A_{ i }^{+}  \right)}
\end{equation}
since $A^i_-$ implies $U^i(\Delta)=0$. Likewise
\begin{align}\notag
\frac{\underset{v:v^i \in I_{ij} }{\inf} \mP^{2\Delta}_{v} \big( N^i (2\Delta,  3\Delta ] >0, C^i_+\big)}{\underset{v:v^i \in I_{ij} }{\sup} \mP^{2\Delta}_{v} \big( C^i_+\big)}
&\leq
\frac {\mP_u \big(\tilde D^{j\to i}(\Delta) \cap C^i_+| C^{ji}_-\big)} {\mP_u \big(C^i_+|C^{ji}_-\big)}
\\\label{MarkovCD}
&\leq
 \frac {\underset{v:v^i \in I_{ij} }{\sup}\mP^{2\Delta}_{v} \big( N^i (2\Delta,  3\Delta ] >0, C^i_+       \big)}{\underset{v:v^i \in I_{ij} }{\inf}\mP^{2\Delta}_{v} \big( C^i_+ \big)}
\end{align}
where $I_{ij}=\{0\}$ if $j \notin \mathcal V_{i}$ , $I_{ij}=[w_{j \to i},+\infty)$ if $j \in \mathcal V_{i}^+$  and $I_{ij}=(-\infty, w_{j \to i}]$ if $j \in \mathcal V_{i}^-$. Indeed $C^{ji}_-$ implies: $U^i(\Delta)=0$ in first case, $U^i(\Delta)\geq w_{j \to i}$ in the second case, and $U^i(\Delta)\leq w_{j \to i}$ in the third case.

 Since, when $j \notin \mathcal {V}^i$, conditionally to $U^i(\Delta)=0$
\begin{equation}\label{keyineq}
\{N^{i, \phi_i(0)} (\Delta,  2\Delta ] >0,\overline {N}^{\mathcal {V}^i}(\Delta,2\Delta]\}\subset \{N^{i} (\Delta,  2\Delta]>0, A^i_+\}
\end{equation}
we have, for any $v$ with $v_i=0$,
\begin{align} \notag
\mP^{\Delta}_{v} \big( N^i (\Delta,  2\Delta ] >0, A^i_{+}         \big)
&\geq\mP \big( N^{i,\leq \phi_i(0)} (\Delta,  2\Delta ] >0, \overline N^{ \mathcal {V}^i}(\Delta,2\Delta]=0 \big)
\\\label{num1}
&\geq  ( 1-\-e^{-\phi_i (0) \Delta}) \-e^{- d_i \beta \Delta}.
\end{align}
On the other hand $\mP^{\Delta}_{v} \left  ( A^i_+ \right    )\leq \-e^{-d_i\alpha \Delta}$, from which, by taking \eqref{MarkovAB} into account, the leftmost inequality in \eqref{PBA} is obtained, with $d_i$ in place of  $d$. 
The same argument can be used to prove that, when $j \notin \mathcal {V}^i$
$$
\mP^{2\Delta}_{v} \left  (N^i (2\Delta,  3\Delta ] >0, C^i_+       \right    )
\geq ( 1-\-e^{-\phi_i (0) \Delta})\-e^{-d_i\beta \Delta}
$$
and together with $\mP^{2\Delta}_{v}(C^i_+)\leq \-e^{-d_i\alpha \Delta}$, the leftmost inequality in \eqref{PDC} is established, with $d_i$ in place of  $d$.

On the other side, always conditionally to $U^i(\Delta)=0$, from the inclusions
\begin{align}\notag
&\{N^{i} (\Delta,  2\Delta]>0, A^i_+\} \subset \{N^{i,\phi_i(0)} (\Delta,  2\Delta ] >0, \underline {N}^{\mathcal {V}^i}(\Delta,2\Delta]=0 \},
\\
\label{otherkey}
& \{ \overline {N}^{\mathcal {V}^i}(\Delta,2\Delta]>0\} \subset A^i_+
\end{align}
and the rightmost inequality in \eqref{MarkovAB}, one obtains
$$
 \frac{\mP_u \big( \tilde B^i(\Delta) \big)}{\mP_u \big( \tilde A^i(\Delta) \big)}
\leq \frac{\mP (N^{i,  \phi_i(0)}(\Delta, 2\Delta)>0)\mP ( \underline {N}^{\mathcal {V}^i}(\Delta,2\Delta]=0)}{\mP (\overline {N}^{\mathcal {V}^i} (\Delta, 2\Delta]=0)}=\frac{(1-\-e^{- \phi_i (0) \Delta})  \-e^{- \alpha d_i \Delta }}{\-e^{- \beta d_i \Delta }}
$$
which is the rightmost inequality in \eqref{PBA}, with $d_i$ in place of  $d$. 

Again, the same argument can be used to prove that
\begin{align*}
\frac{\mP_u \big(\tilde D^{j \to i}(\Delta)\big)}{\mP_u\big(\tilde C^{j \to i}(\Delta)\big)}
&\leq \frac{\mP (N^{i, \phi_i(0)}(2\Delta, 3\Delta)>0) \mP( \underline {N}^{\mathcal {V}^i}(2\Delta, 3\Delta)=0)}{\mP (\overline {N}^{\mathcal {V}^i} (2\Delta, 3\Delta]=0)}
\\&=\frac{(1-\-e^{- \phi_i (0) \Delta})  \-e^{- \alpha d_i \Delta }}{\-e^{- \beta d_i \Delta }},
\end{align*}
which is the rightmost inequality in \eqref{PDC}.

For the proof of \eqref{PDC>} observe that by the Markov property ${C}^{ji}_{-}$ implies $U^i(2\Delta)\geq w_{j \to i}>0$, hence for any $v$ with $v_i \geq w_{j\to i}$
$$
\mP^{2\Delta}_{v} \left  (  N^i (2\Delta,  3\Delta ] >0, C^i_+ \right   )    \geq
(1-\-e^{- (\phi_i (0) + \delta )\Delta}) \-e^ {- \beta d_i \Delta } .
$$
Indeed, conditionally to $U(2\Delta)=v$,
$$
\phi_i(U^i(2\Delta))\geq \phi_i(w_{j\to i})\geq \phi_i(0)+\delta,
$$
and therefore
$$
\{N^i (2\Delta,  3\Delta ] >0, N^{\mathcal {V}^i} (2\Delta,  3\Delta ] =0\}
\supset
\{N^{i, \phi_i(0)+\delta}(2\Delta,  3\Delta ] >0, \overline {N}^{\mathcal {V}^i} (2\Delta,  3\Delta ] =0\}.
$$
With the upper bound
$$\mP^{2\Delta}_{v}(C^i_+)\leq \mP(\underline{N}^{\mathcal{V}^i}(2\Delta, 3 \Delta]=0)= e^{-\alpha d_i \Delta}
$$
the proof of \eqref{PDC>} is finished  with $d_i$ in place of $d$.

For proving \eqref{PDC<}, observe that by the Markov property this time ${C}_{ ij}^{-}$ implies $U^i(2\Delta)\leq w_{j \to i}<0$, hence by using the rightmost inequality of \eqref{MarkovCD} 
\begin{align*}
 \frac{\mP_u \big(\tilde D^{j \to i}(\Delta)\big)}{\mP_u\big(\tilde C^{j \to i}(\Delta)\big)}
 &\leq
\frac {\mP(N^{i,  \phi_i(0)-\delta}(2\Delta,  3\Delta ] >0)\mP(\underline {N}^{\mathcal {V}^i} (2\Delta,  3\Delta ] =0 )}{\mP(\overline {N}^{\mathcal {V}^i} (2\Delta,  3\Delta ] =0 )}
\\& =
\frac { (1- \-e^{- (\phi_i (0) - \delta )\Delta})\-e^{- \alpha d_i \Delta}}{ \-e^{- \beta d_i \Delta}}.
\end{align*}
Indeed, conditionally to  $U(2\Delta)=v$, with  $v^i\leq w_{j \to i}<0$
$$
\phi_i(U^i(2\Delta))\leq \phi_i(w_{j\to i})\leq \phi_i(0)-\delta,
$$
and therefore
$$
\{N^i (2\Delta,  3\Delta ] >0, C^i_+\} \subset
\{N^{i,  \phi_i(0)-\delta}(2\Delta,  3\Delta ] >0,\underline {N}^{\mathcal {V}^i} (2\Delta,  3\Delta ] =0 \},
$$
With the lower bound
$$\mP^{2\Delta}_{v}(C^i_+)\geq \mP(\overline{N}^{\mathcal{V}^i}(2\Delta, 3 \Delta]=0)= e^{-\beta d_i \Delta},
$$
the proof of \eqref{PDC<} is finished with $d_i$ in place of $d$.

\end{proof}

Collecting together the results of the previous two lemmas we arrive to the following

\begin{lemma}\label{probability}
Irrespectively of the vector $u$
of membrane potentials, for $0<\Delta<\Delta_0=\frac{s^2}{5d\beta}$,  it holds
$$
\left( 1- \frac {3d\beta\Delta }{s}\right)\,\phi_i(0)\Delta\leq
\frac {\mP_u ( B^i(\Delta))}{ \mP_u (A^i(\Delta))}\leq
\left (1+\frac {4d\beta \Delta}{s^2}\right)\, \phi_i(0)\Delta.
$$
\\
Furthermore, for $j \notin \mathcal {V}^i$, it holds
$$
\left( 1-\frac {5d\beta \Delta}{s^2} \right) \,\phi_i(0)\Delta \leq
\frac { \mP_u (D^{j\to i}(\Delta))}{ \mP_u (C^{j\to i}(\Delta))}\leq  \left (1+\frac {5d\beta \Delta}{s^3}\right)\,\phi_i(0)\Delta,
$$
whereas, for $j \in \mathcal {V}^i_+$, it holds
$$
\left( 1-\frac {5d\beta \Delta}{s^2} \right) \,(\phi_i(0)+\delta) \Delta \leq \frac {\mP_u ( D^{j\to i}(\Delta))}{ \mP_u (C^{j\to i}(\Delta))},
$$
and finally, for $j \in \mathcal {V}^i_-$, it holds
$$
\frac {\mP_u (D^{j\to i}(\Delta))}{ \mP_u (C^{j\to i}(\Delta))} \leq
\left (1+\frac {5d\beta \Delta}{s^3} \right)\, (\phi_i(0)-\delta) \Delta\,.
$$
\end{lemma}
\begin{proof} As far as the lower bounds are concerned, observe that
\begin{align*}
\left (1-\frac{1-\-e^{-2d\beta\Delta}}{s}\right )\-e^{-d(\beta-\alpha)\Delta}& \geq \left (1-\frac {2d\beta}{s}\Delta \right )\left(1-d\beta (1-s)\Delta \right)
\\
&\geq 1-d\beta \left (\frac {2}{s}+1-s\right ) \Delta \geq 1-\frac {9d\beta}{4s}\Delta>0,
\intertext{
and
}
\left (1-\frac{1-\-e^{-3d\beta\Delta}}{s^{2}} \right)\-e^{-d(\beta-\alpha)\Delta} &\geq \left(1-\frac {3d\beta}{s^2}\Delta\right)(1-d(\beta-\alpha)\Delta)
\\
&\geq 1-d \beta\left (\frac{3}{s^2}+1-s \right )\Delta\geq1-\frac {85d\beta}{27s^2}\Delta>0,
\end{align*}
where these inequalities are guaranteed since $\Delta_0 < \frac {27 s^2}{85d\beta}$.

Next observe that
\begin{equation}\label{zaki}
1-\-e^{-x}\geq x\left( 1-x/2\right)>0,\,\, x\in (0,2),
\end{equation}
so that
\begin{align*}
1-\-e^{-\phi_i(0) \Delta}\geq \phi_i(0) \Delta\,\left( 1-\phi_i(0) \Delta/2\right)\geq \phi_i(0) \Delta\left( 1-\beta \Delta/2\right)>0
\end{align*}
which are guaranteed since $\Delta_0< \frac {2}{\beta}$. As a consequence, given that
$$
\frac{3d\beta}{s} \geq \frac {9d\beta}{4s}+\frac {\beta}{2}, \frac {5d\beta}{s^2} \geq \frac {85d\beta}{27s^2}+\frac {\beta}{2}
$$
we have
\begin{align*}
&\left (1-\frac{1-\-e^{-2d\beta\Delta}}{s}\right )
\-e^{-d(\beta-\alpha)\Delta}\left( 1-\-e^{-\phi_i(0)\Delta}\right)
\\&{}\qquad
\geq \left(1-\frac {9d\beta}{4s}\Delta\right) \left( 1-\frac {\beta \Delta}{2}\right) \phi_i(0) \Delta
\geq \left(1-\frac {3d\beta \Delta}{s}\right)  \phi_i(0) \Delta >0,
\intertext{ and, due to $\phi_i(0)+\delta\leq \beta$}
&\left (1-\frac{1-\-e^{-3d\beta\Delta}}{s^{2}} \right)\-e^{-d(\beta-\alpha)\Delta}\left( 1-e^{-(\phi_i(0)+\epsilon \delta)\Delta}\right)
\\&{}\qquad
\geq \left(1-\frac {85d\beta}{27s^2}\Delta\right)\left( 1-\frac {\beta}{2}\Delta \right)(\phi_i(0)+\epsilon \delta)\Delta
\geq\left(1-\frac {5d\beta \Delta}{s^2}\right)  \left(\phi_i(0)+\epsilon \delta\right) \Delta>0,
\end{align*}
where $\epsilon\in \{0,1\}$, for $\Delta < \Delta_0=\frac {s^2}{5d\beta}$.

Using the bound $1-\-e^{-x} \leq x$ one gets the following upper bounds
\begin{equation}\label{upper1}
\frac {\mP (B^{i}(\Delta))}{ \mP (A^{i}(\Delta))} \leq  \-e^{d\beta(1-s)\Delta}\big(1+\frac{\-e^{2d\beta \Delta}-1}{s^2}\big)\phi(0)\Delta =\frac{ \-e^{d\beta (3-s)\Delta }-(1-s^2)\-e^{d\beta (1-s) \Delta}}{s^2}\phi(0)\Delta
\end{equation}
and
\begin{equation}\label{upper2}
\frac {\mP (D^{j\to i}(\Delta))}{ \mP (C^{j\to i}(\Delta))} \leq \-e^{d\beta(1-s)\Delta}(1+\frac{\-e^{3d\beta \Delta}-1}{s^3})\phi(0)\Delta =\frac{ \-e^{d\beta (4-s)\Delta} -(1-s^3)\-e^{d\beta (1-s) \Delta}}{s^3}\phi(0)\Delta.
\end{equation}

The function $x \to \frac {\-e^x -1}{x}$ is increasing. Therefore, for $0<x\leq \zeta$, we have
$$
\-e^x \leq 1+\frac {\-e^{\zeta}-1}{\zeta}x.
$$

In the case for interest for us the two terms appearing in \eqref{upper1} and \eqref{upper2}  have exponential rates bounded uniformly in $0<s<1$ by
$$
d\beta (3-s) \Delta_0 =\frac {s^2(3-s)}{5}\leq \frac {2}{5}, \quad d\beta (4-s) \Delta_0 =\frac {s^2(4-s)}{5}\leq  \frac {3}{5},
$$
respectively. Then, by taking into account the fact that both $\frac {5}{2}(\-e^{2/5}-1)$ and $\frac {5}{3}(\-e^{3/5}-1)$ do not exceed $3/2$ the expression \eqref{upper1} is bounded from above by
$$
\frac{1+\frac{3}{2}d\beta(3-s)\Delta}{s^2}-\frac{(1-s^{2})(1+d\beta(1-s)\Delta)}{s^2}=1+\frac{d\beta \Delta}{s^2}\left \{\frac {3}{2}(3-s)-(1-s)(1-s^2)\right \}
$$
and the expression \eqref{upper2} is bounded from above by
$$
\frac{1+\frac{3}{2}d\beta (4-s)\Delta}{s^3}- \frac{(1-s^{3})(1+d\beta(1-s)\Delta)}{s^3}=1+\frac{d\beta \Delta}{s^3}\left \{\frac {3}{2}(4-s)-(1-s)(1-s^3)\right \},
$$
with the quantities in brackets bounded by $4$ and $5$, respectively, uniformly for $s\in (0,1)$.
\end{proof}

In order to finish the proof of the theorem we examine the consequences of the bounds established by the previous Lemma \ref{probability} on the quantity
\begin{equation}\label{difference}
\frac { \mP_u \big(D^{j\to i}(\Delta)\big)}{ \mP_u \big(C^{j\to i}(\Delta)\big)}-\frac {\mP_{u^\prime} \big( B^{i}(\Delta)\big)}{ \mP_{u^\prime} \big( A^{i}(\Delta)\big)}.
\end{equation}
In view of the fact that the two terms above will be estimated from data with error, we multiply all the upper bounds established in the previous lemma by $1+\frac {\tau}{10}$ and all the lower bounds by $1-\frac {\tau}{10}$.
Then,
recalling the definitions \eqref{csi1} and \eqref{csi2} of $\xi_1(\Delta)$ and $\xi_2(\Delta)$,
 respectively, it is obtained that, for $0<\Delta<\Delta_0=\frac {s^2}{5d\beta}$:
\\
\indent if $j \notin \mathcal {V}^i$, it holds
\begin{equation}\label{both}
- \xi_1(\Delta) < \frac { \mP_u ( D^{j\to i}(\Delta))}{ \mP_u( C^{j\to i}(\Delta))}-\frac {\mP_{u^\prime} (B^{i}(\Delta))}{ \mP_{u^\prime} (A^{i}(\Delta))}  <  \xi_2(\Delta);
\end{equation}
\indent if $j \in \mathcal {V}^i_-$, then
\begin{equation}\label{inhibitory}
 \frac { \mP_u ( D^{j\to i}(\Delta))}{ \mP_u ( C^{j\to i}(\Delta))}-\frac {\mP_{u^\prime} (B^{i}(\Delta))}{ \mP_{u^\prime} (A^{i}(\Delta))}\leq \xi_2(\Delta) -\tau \lambda_2(\Delta),
\end{equation}
with
\begin{equation*}
\lambda_2(\Delta)=\beta \Delta\,\left[(1+ \frac {5d\beta\Delta}{s^3})(1+\frac {\tau}{10})\right];
\end{equation*}
\indent if $j \in \mathcal {V}^i_+$, then
\begin{equation}\label{excitatory}
\tau \lambda_1(\Delta)- \xi_1(\Delta) \leq
 \frac { \mP_u ( D^{j\to i}(\Delta))}{ \mP_u ( C^{j\to i}(\Delta))}-\frac {\mP_{u^\prime} (B^{i}(\Delta))}{ \mP_{u^\prime} (A^{i}(\Delta))},
\end{equation}
with
\begin{equation*}
\lambda_1(\Delta)=\beta \Delta\,\left[(1-\frac {5d\beta \Delta}{s^2})(1-\frac {\tau}{10})\right].
\end{equation*}
\bigskip

The following lemma allows  to finish with the proof. 
\begin{lemma}\label{messo} If  $0<\Delta \leq \Delta^*=\frac {s^3\tau}{34 d\beta}$, then both the inequalities
\begin{equation*}
\xi_2(\Delta) -\tau \lambda_2(\Delta)\leq- \xi_1(\Delta)
\end{equation*}
and
\begin{equation*}
 \xi_2(\Delta)\leq \tau \lambda_1(\Delta) -\xi_1(\Delta)
\end{equation*}
hold. 
\end{lemma}

\begin{proof} Since $\lambda_1(\Delta)< \lambda_2(\Delta)$ the former inequality is implied by the  latter one,
which is satisfied for
$$
0<\Delta \leq \frac {s^3\tau (6-\tau)}{d\beta\{(4s+5)(10+\tau)+s(10-\tau)(3s+5(1+\tau))\}}.
$$

In order to prove that the right hand side is bounded above by $ \Delta^*$, notice that the numerator is bounded from below by $5 s^3 \tau$, whereas 170 is the maximum value of the expression within brackets at the denominator for $s, \tau \in [0,1]$, with $s+\tau \leq 1$,  attained at $s=1, \tau=0$.
\end{proof}

\section{Proof of Theorem \ref{main}}\label{proof2}

In the proof of Theorem \ref{main},  we will use the following two inequalities.
\begin{prop}\label{bennett}
Let $X$ be a random variable with binomial distribution, with parameters $n$ and $p$, and let $0<\gamma<1$. Then
$$
\mathbb P(X\leq np(1-\gamma)) \leq \-e^{-np\gamma^2/2}
$$
$$
\mathbb P(X\geq np(1+\gamma)) \leq \-e^{-np\gamma^2/3}
$$
\end{prop}
For the proof of Proposition \ref{bennett} we refer the reader to \cite{hagerup}.

\begin{corollary}\label{domination}
Let $Y_1,\ldots,Y_n$ be Bernoulli random variables with the property
\begin{equation}\label{caso-c}
\mathbb P(Y_{k+1}=1|Y_1,...,Y_k)\geq c,\,\,\, k=0,\ldots,n-1,
\end{equation}
for some constant $c>0$. Then for any $0<\gamma<1$
\begin{equation}\label{Tesicaso-c}
\mathbb P(Y_1+\ldots+Y_n\leq nc(1-\gamma)) \leq \-e^{-nc\gamma^2/2}.
\end{equation}

When $Y_1,\ldots,Y_n$ have the property
\begin{equation}\label{caso-C}
\mathbb P(Y_{k+1}=1|Y_1,...,Y_k)\leq C,\,\,\, k=0,\ldots,n-1,
\end{equation}
for some constant $C<1$, then for any $0<\gamma<1$
\begin{equation}\label{Tesicaso-C}
\mathbb P(Y_1+\ldots+Y_n\geq nC(1+\gamma)) \leq \-e^{-nC\gamma^2/3}.
\end{equation}
\end{corollary}
\begin{proof}
 Let us denote
$$
p_{k+1}(y_1,...,y_k)=P(Y_{k+1}=1|Y_1=y_1,...,Y_k=y_k),\quad k=0,...,n-1,
$$
where $y_i\in \{0,1\}$, $i=1,...,k$.\\

With $U_1$,...,$U_n$  independent and uniformly distributed in $(0,1)$, the random variables $Y^\prime_k$, $k=1,...,n$, are  constructed recursively as
$$
\begin{cases}
Y^\prime_{k+1}= \mathbf{1}_{[0,p_{k+1}(Y^\prime_1,...,Y^\prime_{k})]}(U_{k+1}), & \quad k=1,...,n-1, \vspace{2mm}
\\
Y^\prime_1=\mathbf{1}_{[0,p_{1}]}(U_{1}),&
\end{cases}
$$
so that $(Y^\prime_1,...,Y^\prime_n)$ and $(Y_1,...,Y_n)$ share the same distribution.
Moreover when \eqref{caso-c} holds then
$Y^\prime_k\geq \mathbf{1}_{[0,c]}(U_{k})$,  while when \eqref{caso-C} holds then
$Y^\prime_k\leq \mathbf{1}_{[0,C]}(U_{k})$, for $k=1,...,n$. The proof is completed by observing the random variables $\mathbf{1}_{[0,v]}(U_{k})$, $k=1,...,n$ are Bernoulli i.i.d.\@, for any fixed $v\in(0,1)$,  and the application of Proposition \ref{bennett}: in the former case the law of $Y_1+\ldots+Y_n$ stochastically dominates the binomial distribution with parameters $n$ and $c$, and in the latter is stochastically dominated by the binomial distribution with parameters $n$ and $C$.
\end{proof}

We can now prove Theorem \ref{main} through a series of lemmas. For the first one  we recall the definitions \eqref{parameters}
$$
t_n=\left\lceil \alpha \Delta^{*}n \right\rceil ,\, \, m_n=\left\lceil 
\frac {19}{20}n  \, \alpha^2\Delta^{*2}(1-\frac {\tau}{10} \sqrt {\alpha \Delta^{*}})\right\rceil.
$$

\begin{lemma}\label{howmany}
For $i,j \in I$, and any integer $n \ge 1$
the following inequalities hold 
\begin{equation}\label{firstlemma0}
\mP \left (S^{A_i}(t_n) < m_n\right )\leq
\rho(n),
\end{equation}
\begin{equation}\label{firstlemma1}
\mP \left (S^{C^{j \to i}}(n) < m_n  \right ) \leq
\rho(n),
\end{equation}
where
\begin{equation}\label{rho}
\rho(n)=  \-e^{-\frac{19}{4 \times 10^3} \alpha^3 \Delta^{*3} \tau^2n}\, .
\end{equation}
\end{lemma}
\begin{proof}
Since for any interval $I$, $N^i(I)\geq \underline N^{i}(I)$, we have the following lower bounds
$$
S^{A_i}(t_n)\geq \sum_{k=1}^{t_n} \mathbf{1}_{\{ \underline N^{i}((2k-2)\Delta, (2k-1)\Delta]>0\}}=:{\underline S}^{A_i}(t_n),
$$
$$
S^{ C^{j\to i}}(n)\geq \sum_{k=1}^{\ell} \mathbf{1}_{ \{\underline N^{i}((3k-3)\Delta, (3k-2)\Delta]>0\}}\mathbf{1}_{\{ \underline N^{j}((3k-2)\Delta, (3k-1)\Delta]>0\}}=:{\underline S}^{ C^{j\to i}}(n),
$$
where the two variables at the r.h.s. are binomial with $t_n$ trials and $n$ trials, respectively, and probability of success bounded from below by 
\begin{equation}\label{prima}
1-\-e^{-\alpha \Delta^*}
\ge \left (1-\frac {s^4\tau}{68d} \right )\alpha \Delta^* 
\ge \frac {67}{68} \alpha \Delta^* 
>\frac {19}{20} \alpha \Delta^*
\end{equation} 
and
\begin{equation}\label{seconda}
(1-\-e^{-\alpha \Delta^*})^2 
\ge \left  (\frac {67}{68} \right )^2 \alpha^2 \Delta^{*2} > \frac {19}{20} \alpha^2 \Delta^{*2},
\end{equation}
respectively, by using \eqref{zaki}. The bounds \eqref{firstlemma0} and \eqref{firstlemma1} are then obtained by applying Proposition \ref{bennett}.  Indeed 
$$
\mP \left (S^{A_i}(t_n) < m_n\right )\leq \mP \left (S^{A_i}(t_n) \leq  \frac {19}{20}\alpha^2 \Delta^{*2} (1-\frac {\tau}{10} \sqrt {\alpha \Delta^{*}})  \cdot  n \right ) 
$$
$$
\leq \mP \left (S^{A_i}(t_n) \leq  \frac {19}{20}\alpha\Delta^* (1-\frac {\tau}{10} \sqrt {\alpha \Delta^{*}})  \cdot t_n    \right )  
\leq \mP \left ({\underline S}^{A_i}(t_n) \leq  \frac {19}{20}\alpha\Delta^* (1-\frac {\tau}{10} \sqrt {\alpha \Delta^{*}}
) \cdot t_n \right ) 
$$
$$
 \leq \-e^{-\frac {19}{4 \times 10^3}\alpha^2 \Delta^{*2}\tau^2 t_n}\leq \-e^{-\frac {19}{4 \times 10^3}\alpha^3 \Delta^{*3}\tau^2 n},
$$
and
$$
\mP \left (S^{ C^{j\to i}}(n)<m_n \right) \leq \mP \left (S^{ C^{j\to i}}(n) \leq  \frac {19}{20}\alpha^2\Delta^{*2} (1-\frac {\tau}{10} \sqrt {\alpha \Delta^{*}}) \cdot n \right )
$$
$$
\leq \mP \left ({\underline S}^{ C^{j\to i}}(n) \leq  \frac {19}{20}\alpha^2\Delta^{*2} (1-\frac {\tau}{10} \sqrt {\alpha \Delta^{*}})     \cdot n \right )
 \leq \-e^{-\frac {19}{4 \times 10^3}\alpha^3 \Delta^{*3}\tau^2 n}.
$$
\end{proof}

\begin{lemma}\label{prob-cond} For any positive integer $n$ define
\begin{equation}\label{sigma-ell-gamma}
\sigma(n)=\-e^{-\frac{19^2}{116 \times 10^3} \alpha^3 \Delta^{*3}\tau^2 n} .
\end{equation}
 Then
\begin{equation}\label{secondlemma1}
\mP \left (S^{B^i}(K^i_{m_n})\leq m_n\phi_i(0)\Delta^*
\left (1-\frac {3d\beta \Delta^*}{s} \right ) \left (1-\frac {\tau}{10} \right ) \right) \leq \sigma(n),
\end{equation}
\begin{equation}\label{secondlemma2}
\mP \left (S^{B^i}(K^i_{m_n})\geq m_n\phi_i(0)\Delta^* 
\left (1+\frac {4d\beta \Delta^*}{s^2} \right ) \left (1+ \frac {\tau}{10} \right ) \right) \leq \sigma(n);
\end{equation}
\\
moreover, if $j \notin \mathcal {V}^i$
\begin{equation}\label{thirdlemma1}
\mP \left (S^{D^{j\to i}}(H^{j\to i}_{m_n})\leq m_n\phi_i(0)\Delta^* \left (1-\frac {5d \beta \Delta^*}{s^2} \right ) \left (1- \frac {\tau}{10}  \right ) \right)
\leq 
\sigma(n),
\end{equation}
\begin{equation}\label{thirdlemma2}
\mP \left (S^{D^{j\to i}}(H^{j\to i}_{m_n})\geq m_n \phi_i(0)\Delta^* \left  (1+\frac {5d \beta \Delta^*}{s^3} \right ) \left (1+ \frac {\tau}{10} \right )\right)\leq \sigma(n);
\end{equation}
whereas if $j \in \mathcal {V}^-_i$
\begin{equation}\label{fourthlemma}
\mP \left (S^{D^{j\to i}}(H^{j\to i}_{m_n})\geq m_n(\phi_i(0)-\delta)\Delta^* \left (1+\frac {5d \beta \Delta^*}{s^3} \right)\left (1+\frac {\tau}{10} \right )\right) \leq
\sigma(n); 
\end{equation}
and if $j \in \mathcal {V}^+_i$
\begin{equation}\label{fifthlemma}
\mP \left (S^{D^{j\to i}}(H^{j\to i}_{m_n})\leq m_n(\phi_i(0)+\delta)\Delta^* \left (1-\frac {5d \beta \Delta^*}{s^2} \right )            \left (1-\frac {\tau}{10}    \right )\right)\leq \sigma(n).
\end{equation}
\end{lemma}
\begin{proof}
The estimates are obtained by means of Corollary~\ref{domination} with the choices
$$
Y_h'=\mathbf{1}_{B^i_{K^{i}_{h}}},\, h=1,2,\ldots,m_n, \qquad Y_k''=\mathbf{1}_{{D}^{j\to i}_{H^{j\to i}_{k}}},\,\, k=1,2,\ldots,m_n,
$$
respectively. Indeed, observe that
$$
S^{B^i}(K^i_{m_n})=\sum_{h=1}^{m_n}Y_h', \quad \text{and}
\quad
S^{D^{j\to i}}(H^{j\to i}_{m_n})=\sum_{k=1}^{m_n}Y_k''.
$$

For any non negative integer $h$ and for any value of $j_1,...,j_h\in \{0,1\}$

\begin{align*}
 &\mP(Y'_{h+1}=1|Y'_1=j_1,...,Y'_h=j_h)
\\
=&\sum_{\ell=h+1}^{\infty} \mP (K^i_{h+1}=\ell|Y'_1=j_1,...,Y'_h=j_h)
\mP(B^i_{\ell}|Y'_1=j_1,...,Y'_h=j_h,K^i_{h+1}=\ell).
\end{align*}
Notice that
$$
\{K^i_{h+1}=\ell\}=\{ K^i_h\leq \ell-1, K^i_{h+1}>\ell-1, A^i_{\ell}\}
$$
and that the event
$$
F=\{Y'_1=j_1,...,Y'_h=j_h, K^i_h\leq \ell-1, K^i_{h+1}>\ell-1\}
$$
is  $\calF_{2(\ell-1)\Delta^*}$-measurable, so that we get
\begin{align*}
&\mP(B^i_{\ell}|Y'_1=j_1,...,Y'_h=j_h,K^i_{h+1}=\ell)=\mP( B^i_{\ell}|F\cap A^i_{\ell})
=\frac{\mathbb{E}\big[\mP(B^i_{\ell}|\calF_{2(\ell-1)\Delta^*}) \mathbf{1}_F\big]}{\mathbb{E}\big[\mP(A^i_{\ell}|\calF_{2(\ell-1)\Delta^*}) \mathbf{1}_F\big]}.
\end{align*}

Using again the notation $\mP^{k\Delta^*}_{v}(\cdot)=\mP(\cdot|U(k\Delta^* )  = v)$,  for $k\geq 1$, we have
\begin{align*}
\mP(B^i_{\ell}|\calF_{2(\ell-1)\Delta^*})=\mP^{2(\ell-1)\Delta^*}_{U(2(\ell-1)\Delta^*)}(B^i_{\ell})=\mP_{U(2(\ell-1)\Delta^*)}(B^i_1), \end{align*}
\begin{align*}
\mP(A^i_{\ell}|\calF_{2(\ell-1)\Delta^*})=\mP^{2(\ell-1)\Delta^*}_{U(2(\ell-1)\Delta^*)}(A^i_{\ell})=\mP_{U(2(\ell-1)\Delta^*)}(A^i_1).
\end{align*}

As a consequence the bounds in Lemma \ref{probability} can be applied, from which, for $0\leq h\leq m_n-1$,
\begin{equation}\label{inserito1}
\left( 1- \frac {3d\beta\Delta^* }{s}\right)\,\phi_i(0)\Delta^*\leq
\mP(Y'_{h+1}=1|Y'_1=j_1,...,Y'_h=j_h)\leq
\left (1+\frac {4d\beta \Delta^*}{s^2}\right)\, \phi_i(0)\Delta^*,
\end{equation}
and when $j \notin \mathcal {V}^i$
\begin{equation}\label{inserito2}
\left( 1- \frac {5d\beta\Delta^* }{s}\right)\,\phi_i(0)\Delta^*\leq
\mP(Y''_{h+1}=1|Y''_1=j_1,...,Y''_h=j_h)\leq
\left (1+\frac {5d\beta \Delta^*}{s^2}\right)\, \phi_i(0)\Delta^*,
\end{equation}
whereas when $j \in \mathcal {V}^i_+$
\begin{equation}\label{inserito3}
\left( 1- \frac {5d\beta\Delta^* }{s}\right)\,(\phi_i(0)+\delta)\Delta^*\leq
\mP(Y''_{h+1}=1|Y''_1=j_1,...,Y''_h=j_h)
\end{equation}
and finally, when $j \in \mathcal {V}^i_-$
\begin{equation}\label{inserito4}
\mP(Y''_{h+1}=1|Y''_1=j_1,...,Y''_h=j_h)\leq
\left (1+\frac {5d\beta \Delta^*}{s^2}\right)\, (\phi_i(0)-\delta)\Delta^*.
\end{equation}
Now one applies Corollary \ref{domination} to all these bounds, with $m_n$ in place of $n$, and $\gamma =\frac {\tau}{10}$.

Beginning with the leftmost inequality in \eqref{inserito1}, with $c=( 1- 3d\beta\Delta^* /s)\phi_i(0)\Delta^*$ in \eqref{Tesicaso-c}, we obtain
\begin{align}\notag
&\mP \left (S^{B^i}(K^i_{m_n})\leq m_n\phi_i(0)\Delta^*
\left (1-\frac {3d\beta \Delta^*}{s} \right ) \left (1-\frac {\tau}{10} \right ) \right) \leq \-e^{-\frac{1}{2}m_n \phi_i(0)\Delta^* (1-\frac {3d\beta \Delta^*}{s})(\frac{\tau}{10})^2}
\\
\leq &{} \-e^{- \frac{1}{2}m_n \alpha\Delta^* (1-\frac {3d\beta \Delta^*}{s})(\frac{\tau}{10})^2}
\leq \-e^{- n \alpha^3 \Delta^{*3}(1-\frac {\tau}{10} \sqrt {\alpha \Delta^{*}}) (\frac{\tau}{10})^2\frac{19 \times 31}{40 \times 34}}
\leq \-e^{- n \alpha^3 \Delta^{*3}{\tau}^2\frac{3 \times 19^2 \times 31}{4 \times 34 \times 58 \times 10^3}}=\sigma(n)^{\frac {31}{34}\times \frac {3}{2}}.\label{first-bound}
\end{align}
In the first inequality at the last line, after replacing $m_n$ with the argument of the integer part, we have taken into account that
$$
\frac{d\beta\Delta^*}{s}=
 \frac{s^{2} \tau}{34}\leq  \frac{1}{34} \,\,\,\, \Rightarrow \,\,\,\, 1-\frac {3d\beta \Delta^*}{s} \ge\frac {31}{34},
$$
and in the second inequality that
$$
\alpha \Delta^{*} =\frac {s^4\tau}{34d} \leq \frac{1}{34} \,\,\,\,\, \Rightarrow \,\,\,\, 1-\frac {\tau}{10} \sqrt {\alpha \Delta^{*}}\geq 1- \frac {1}{10\sqrt{34}}>\frac {57}{58}=\frac {3 \times 19}{58}.
 $$
For the rightmost inequality in \eqref{inserito1} choose $C=(1+4d\beta \Delta^*/s^2)\phi_i(0)\Delta^*$ in \eqref{Tesicaso-C}  obtaining
\begin{align}
&\mP \left (S^{B^i}(K^i_{m_n})\geq m_n\phi_i(0)\Delta^* \left (1+\frac {4d\beta \Delta^*}{s^2} \right ) \left (1+ \frac {\tau}{10} \right ) \right) \leq \-e^{-   \frac{1}{3} m_n \phi_i(0)\Delta^* (1+\frac {4d\beta \Delta^*}{s^2})(\frac{\tau}{10})^2}
\\
\leq &{} \-e^{- \frac{1}{3}m_n \alpha\Delta^* (1+\frac {4d\beta \Delta^*}{s^2})(\frac{\tau}{10})^2}
\leq \-e^{- n \alpha^3 \Delta^{*3}(1-\frac {\tau}{10} \sqrt {\alpha \Delta^{*}}) (\frac{\tau}{10})^2\frac{19}{60}} \leq \-e^{- n \alpha^3 \Delta^{*3} \tau^2\frac{19^2 }{116 \times 10^3}}=\sigma (n). \label{second-bound}
\end{align}
Taking into account that $\frac {31}{34} \times \frac {3}{2}>1$, 
the estimates~\eqref{secondlemma1} and~\eqref{secondlemma2} are obtained.

Analogously, from Corollary \ref{domination} with $c=( 1- 5d\beta\Delta^* /s)\phi_i(0)\Delta^*$ in \eqref{Tesicaso-c}, for the leftmost inequality in \eqref{inserito2} one obtains
$$
\mP \left (S^{D^{j\to i}}(H^{j\to i}_{m_n})\leq m_n\phi_i(0)\Delta^* \left (1-\frac {5d\beta \Delta^*}{s} \right ) \left (1-\frac {\tau}{10} \right ) \right) \leq 
\-e^{- n \alpha^3 \Delta^{*3}{\tau}^2\frac{3 \times 19^2 \times 29}{4 \times 34 \times 58 \times 10^3}}=\sigma(n)^{\frac {29}{34}\times \frac {3}{2}}
$$
and for the rightmost one, with $C=( 1+5d\beta\Delta^* /s)\phi_i(0)\Delta^*$ in \eqref{Tesicaso-C}, one obtains
$$
\mP \left (S^{D^{j\to i}}(H^{j\to i}_{m_n})\geq m_n\phi_i(0)\Delta^* \left (1+\frac {5d\beta \Delta^*}{s^2} \right ) \left (1+ \frac {\tau}{10} \right ) \right)
\leq \-e^{- n \alpha^3 \Delta^{*3} \tau^2\frac{19^2 }{116 \times 10^3}}=\sigma (n).
$$
Since $\frac {29}{34} \times \frac {3}{2}>1$ the estimates~\eqref{thirdlemma1} and~\eqref{thirdlemma2} are obtained. 
 The bounds \eqref{fourthlemma} and \eqref{fifthlemma} are obtained in a completely analogous way, taking into account that $\phi_i(0)\pm \delta \geq \alpha$.
\end{proof}

\medskip 

With the help of the previous results we are in a position to control the behaviour of the estimators $ R^i(n)$ and $G^{j\to i}(n)$ defined in
\eqref{erre-1-2} and \eqref{gi-1-2}, respectively.
\begin{lemma}\label{ratioABleft}
For any positive integer $n$, the following inequalities hold, with $\sigma(n)$ defined in \eqref{sigma-ell-gamma},
\begin{equation}\label{tag1}
\mP \left ( R^i(n)\leq \phi_i(0)\Delta^* \left (1-\frac {3d\beta \Delta^*}{s}  \right ) \left (1-\frac {\tau}{10} \right ) \right )
\leq 2 \sigma(n)
\end{equation}
\begin{equation}\label{tag2}
\mP \left (  R^i(n)\geq \phi_i(0)\Delta^* \left (1+\frac {4d\beta \Delta^*}{s^2} \right ) \left (1+\frac {\tau}{10} \right ) \right )
\leq 2\sigma(n);
\end{equation}
furthermore:\,  i) if $j \notin \mathcal {V}^i$
\begin{equation}\label{ldevest2a}
\mP \left (G^{j\to i}(n) \leq {\phi_i(0)\Delta^* } \left (1-\frac {5d\beta \Delta^*}{s^2} \right ) \left (1-\frac {\tau}{10} \right ) \right )
\leq 
2\sigma(n)
\end{equation}
\begin{equation}\label{ldevest2b}
\mP \left ( G^{j\to i}(n) \geq {\phi_i(0) \Delta^* } \left (1+\frac {5d\beta \Delta^*}{s^3} \right ) \left (1+\frac {\tau}{10} \right ) \right )
\leq 2\sigma(n)
\end{equation}

ii) if $j \in \mathcal {V}^i_-$
\begin{equation}\label{ldevest3}
\mP \left (G^{j\to i}(n) \geq (\phi_i(0)-\delta)\Delta^* \left (1+\frac {5d\beta \Delta^*}{s^3} \right ) \left (1+\frac {\tau}{10} \right ) \right)\leq
2\sigma(n)
\end{equation}

iii) if $j \in \mathcal {V}^i_+$
\begin{equation}\label{ldevest4}
\mP \left (G^{j\to i}(n) \leq (\phi_i(0)+\delta)\Delta^* \left (1-\frac {5d\beta \Delta^*}{s^2} \right ) \left (1-\frac {\tau}{10} \right )\right)\leq
2\sigma(n).
\end{equation}

\end{lemma}
\begin{proof}
We are going to prove only \eqref{tag1} in detail, since the other inequalities \eqref{tag2} -- \eqref{ldevest4} need completely similar arguments. So, observe that
\begin{align*}
& \left \{R^i(n)\leq \phi_i(0)\Delta^* \left (1-\frac {3d\beta \Delta^*}{s} \right ) \left (1-\frac {\tau}{10} \right ) \right \}
\\&{}\quad  \subset  \left \{ K^i_{m_n}> t_n  \right \}
 \cup \left  \{
S^{B_i}(K^i_{m_n})\leq m_{n} \phi_i(0)\Delta^* \left (1-\frac {3d\beta \Delta^*}{s}  \right ) \left (1-\frac {\tau}{10} \right )  \right \}\, 
\end{align*}
$$
\subset  \left \{S^{A_i}(t_n) < m_n \right \} \cup  \left \{
S^{B_i} (K^i_{m_n})\leq m_{n} \phi_i(0)\Delta^* \left (1-\frac {3d\beta \Delta^*}{s}  \right ) \left (1-\frac {\tau}{10} \right )  \right \}\,.
$$
Since the probability of the two events have been bounded from above by $\rho(n)$ and $\sigma (n)$ in \eqref{firstlemma0} and \eqref{secondlemma1}, respectively, then
\begin{align*}
\mP \left ( R^i(n)\leq \phi_i(0)\Delta^*\left(1-\frac {\tau}{10}\right)\left(1-\frac {3d\beta \Delta^*}{s} \right) \right )
&\leq \rho(n)+\sigma(n)\, .
\end{align*}
 Since $\rho(n)=\sigma(n)^{19/29}<\sigma(n)$, see  \eqref{rho} and \eqref{sigma-ell-gamma}, the proof of \eqref{tag1} is concluded.
\end{proof}
\bigskip

To conclude the proof of Theorem \ref{main} we need to deduce suitable bounds for
 the difference $G^{j\to i}(n)-R^i(n)$. Before stating them it is convenient to recall that
$$
\xi_1(\Delta^*)=\beta \Delta^* \left[\tfrac {\tau}{5}+\left(9-\tfrac {\tau}{10}\right)\frac {d\beta\Delta^*}{s^2}\right], \; \xi_2(\Delta^*)=\beta \Delta^* \left\{\tfrac {\tau}{5}+\left[5+3s^2+\tfrac {\tau(5-3s^2)}{10}\right]\frac {d\beta\Delta^*}{s^3}\right\},
$$
$$
\lambda_1(\Delta^*)=\beta \Delta^*\,\left(1-\frac {5d\beta \Delta^*}{s^2}\right)
\left(1-\tfrac {\tau}{10}\right),\,\,\,\lambda_2(\Delta^*)=\beta \Delta^*\,
\left(1+ \frac {5d\beta\Delta^*}{s^3}\right)
\left(1+\tfrac {\tau}{10}\right).
$$

 \begin{lemma}\label{ultimo} For any positive integer $n$
the following inequalities hold:
\\

\indent
i) if $j \notin \mathcal {V}^i$
\begin{equation}\label{ldevest2aBIS}
\mP \left (G^{j\to i}(n)-R^i(n)\leq - \xi_1(\Delta^*)\right) \leq 4\sigma(n),
\end{equation}
\begin{equation}\label{ldevest2bBIS}
\mP \left (G^{j\to i}(n)-R^i(n) \geq  \xi_2(\Delta^*)\right)
\leq  4\sigma(n),
\end{equation}
from which
\begin{equation}\label{atag}
\mP \left(- \xi_1(\Delta^*)< G^{j\to i}(n)-R^i(n) <  \xi_2(\Delta^*)\right)
\geq  1-8 \sigma(n);
\end{equation}

ii) if $j \in \mathcal {V}^i_-$
\begin{equation}\label{ldevest3BIS}
\mP \left (G^{j\to i}(n)-R^i(n) < \xi_2(\Delta^*) -\tau \lambda_2(\Delta^*)\right)\geq
1-  4\sigma(n\,);
\end{equation}

iii) if $j \in \mathcal {V}^i_+$
\begin{equation}\label{devest4BIS}
\mP \left (G^{j\to i}(n)-R^i(n) >  -\xi_1(\Delta^*) +\tau \lambda_1(\Delta^*)\right)
\geq 1-  4\sigma(n).
\end{equation}
\end{lemma}

\begin{proof} First observe that for $\chi=0,1$, by the union bound
$$
\mP \left (G^{j\to i}(n)-R^i(n)\leq - \xi_1(\Delta^*)+\chi\tau \lambda_1(\Delta^*)\right)
$$
$$ 
\leq \mP \Big(G^{j\to i}(n)-R^i(n)\leq -\frac{\phi_i(0)}{\beta} \xi_1(\Delta^*)+\chi\tau \lambda_1(\Delta^*)\Big)
$$
$$
\leq \mP \left(G^{j\to i}(n)\leq \phi_i(0)\left(1-\frac {5d\beta \Delta^*}{s^2}\right)\left(1-\tfrac {\tau}{10}\right) +\chi\tau \lambda_1(\Delta^*)\right)
$$
$$
+\mP \left(R^i(n)\geq \phi_i(0)\left(1+\frac {4d\beta \Delta^*}{s^2}\right)\left(1+\tfrac {\tau}{10}\right)\right),
$$
since 
$$
 \phi_i(0)\left(1-\frac {5d\beta \Delta^*}{s^2}\right)\left(1-\tfrac {\tau}{10}\right)-\phi_i(0)\left(1+\frac {4d\beta \Delta^*}{s^2}\right)\left(1+\tfrac {\tau}{10}\right)=-\frac{\phi_i(0)}{\beta}\xi_1(\Delta^*).
$$
As a consequence \eqref{ldevest2aBIS} and \eqref{devest4BIS} are established by using the  bounds \eqref{ldevest2a}, \eqref{ldevest4} and \eqref{tag2}. Analogously, for $\chi=0,1$, by the union bound
$$
\mP \left(G^{j\to i}(n)-R^i(n) \geq  \xi_2(\Delta^*)-\chi\tau \lambda_2(\Delta^*)\right)
$$
$$ 
\leq \mP \Big(G^{j\to i}(n)-R^i(n)\geq \frac{\phi_i(0)}{\beta}\xi_2(\Delta^*)-\chi\tau \lambda_2(\Delta^*)\Big)
$$
$$
\leq \mP \left( G^{j\to i}(n)\geq \frac{\phi_i(0)}{\beta}\left(1+\frac {5d\beta \Delta^*}{s^3}\right)\left(1+\tfrac {\tau}{10}\right) -\chi\tau \lambda_2(\Delta^*)\right)
$$
$$
+\mP \left( R^i(n)\leq \frac{\phi_i(0)}{\beta}\left(1-\frac {3d\beta \Delta^*}{s^2} \right) \left(1-\tfrac {\tau}{10}\right)\right),
$$
since 
$$
\frac{\phi_i(0)}{\beta}\left(1+\frac {5d\beta \Delta^*}{s^3}\right)\left(1+\tfrac {\tau}{10}\right)
-\frac{\phi_i(0)}{\beta}\left(1-\frac {3d\beta \Delta^*}{s^2} \right) \left(1-\tfrac {\tau}{10}\right)
=\frac{\phi_i(0)}{\beta}\xi_2(\Delta^*).
$$
 As a consequence \eqref{ldevest2bBIS} and \eqref{ldevest3BIS} are established by using the bounds \eqref{ldevest2b}, \eqref{ldevest3} and \eqref{tag1}. Moreover \eqref{atag}  is trivially obtained by 
\eqref{ldevest2aBIS} and \eqref{ldevest2bBIS}.
\end{proof}

The proof of Theorem~\ref{main} follows directly from  Lemma \ref{messo} and Lemma~\ref{ultimo} and finally by substituting 
$\alpha {\Delta^{*}}=\frac{s^4\tau}{34d}$ in the expression for $\sigma(n)$. The value
$
\vartheta_0=\frac{19^2}{3\times 116\times {34}^2 \times 10^3}$
is readily computed.

\section*{Acknowledgments}
This work is part of USP project {\em Mathematics, computation, language and the brain}, FAPESP project {\em Research, Innovation and Dissemination Center for Neuromathematics} (grant 2013/07699-0),  Sapienza Project 2016 {\em Processi stocastici teoria e applicazioni} RM11615501013C24, and Sapienza Project 2017 {\em Modelli stocastici nelle scienze e nell'in\-ge\-gne\-ria} RM11715C7D9F7762. AG is partially supported by CNPq fellowship (grant 309501/2011-3.).

%\bibliographystyle{abbrv}
%\bibliographystyle{plain}
%\bibliography{Biblio}

\end{document}